\documentclass{amsart}

\usepackage{amsmath}
\usepackage{amssymb}

\newcommand{\Z}{\ensuremath{\mathbb{Z}}}
\newcommand{\To}{\rightarrow}
\newcommand{\N}{\ensuremath{\mathbb{N}}}
\newcommand{\mc}[1]{\ensuremath{\mathcal #1}}

\newcommand{\la}{\ensuremath{\lambda}}
\newcommand{\abs}[1]{\ensuremath{\lvert #1\rvert}}

\def\Super{{\rm Super}}
\def\Pl{{\rm Pl}}
\def\FX{{F\langle X \rangle}}
\def\FL{{F\langle L \rangle}}
\def\L{\mc{M}(L)}

\def\tila{\tilde{\la}}
\def\titau{\tilde{\tau}}
\def\til{\tilde{l}}
\def\char{{\rm char}}

\newtheorem{thm}{Theorem}[section]
\newtheorem{prop}[thm]{Proposition}
\newtheorem{lem}[thm]{Lemma}
\newtheorem{cor}[thm]{Corollary}

\newtheorem{ex}[thm]{Example}

\theoremstyle{definition}
\newtheorem{df}[thm]{Definition}

\theoremstyle{remark}
\newtheorem{rem}[thm]{Remark}

\begin{document}

\title{Super RSK-algorithms and super plactic monoid}

\author[R. La Scala]{Roberto La Scala$^*$}

\address{$^*$ Dipartimento di Matematica,
via Orabona 4, 70125 Bari, Italia}

\email{lascala@dm.uniba.it}
\author[V. Nardozza]{Vincenzo Nardozza$^{\dagger}$}

\address{$^{\dagger}$ Dipartimento di Matematica e Applicazioni,
via Archirafi 34, 90123 Palermo, Italia}

\email{vickkk@tiscali.it}

\author[D. Senato]{Domenico Senato$^{\ddagger}$}

\address{$^{\ddagger}$ Dipartimento di Matematica,
Campus Macchia Romana, 85100 Potenza, Italia}

\email{senato@unibas.it}
\thanks{Partially supported by Universit\`a della Basilicata,
Universit\`a di Bari and COFIN MIUR 2003
}

\begin{abstract}
We construct the analogue of the plactic monoid for the super
semistandard Young tableaux over a signed alphabet. This is done
by developing a generalization of the Knuth's relations. Moreover
we get generalizations of Greene's invariants and Young-Pieri
rule. A generalization of the symmetry theorem in the signed case
is also obtained. Except for this last result, all the other
results are proved without restrictions on the orderings of the
alphabets.
\end{abstract}

\maketitle

\section{Introduction}
Young tableaux are simple objects with a wide range of applications.A vast literature dealing with Young tableaux has been developing through the decades. In our paper, we shall be mainly concerning their relations with an algebraic structure called {\em plactic monoid}.

Starting from the beginning, in 1938 Robinson worked out an
algorithm in order to compute the coefficients of products of
Schur functions (Littlewood-Richardson rule). In 1961 Schensted
(\cite{Schensted61}) brought new life to Robinson algorithm, in a
clearer form. Finally, in 1970 Knuth refined the
Robinson-Schensted correspondence, by detecting its fundamental
laws and giving it the form of a computer program. This, together
with the further results obtained by Greene, revealed the inner
structure of the correspondence, and are the key step in order to
apply it as a combinatorial tool.

Knuth relations were the crucial point in the construction by
Lascoux and Sch\"utzenberger (\cite{LascouxSchutz81}). They turned
the set of Young tableaux into a monoid structure, called the
plactic monoid, taking into account most of their combinatorial
properties and having a wide range of applications.

In these classical settings, the filling of the Young tableaux are
from an alphabet, i.e.  from a totally ordered set $L$. In our
paper we shall use entries from a signed alphabet. A signed
alphabet is a totally ordered set $L$, usually finite or
countable, which is disjoint union of two subsets, say $L_0,L_1$.
We denote $|x| = \alpha$ when $x\in L_\alpha$. Algebraic
structures insisting on a signed alphabet have been conceived to
provide combinatorial methods for important algebraic theories as
the invariant theory in superalgebras \cite{GrosshansRotaStein87}
and the representation theory of general Lie superalgebras
\cite{BereleRegev87}. Let us mention that these theories can be
applied in particular to study algebras satisfying polynomial
identities (see for example \cite{BereleRegev83,CariniDrensky94}).
The combinatorics involved with this subject primarily concern
with Young tableaux. A notion of semistandardness has to be
properly defined for these Young tableaux. Many papers follow the
approach introduced in \cite{BereleRemmel85} by defining these
tableaux as ``$(k,l)$-semistandard tableaux'' that is assuming
$L_0 < L_1$. Following the setting suggested in
\cite{GrosshansRotaStein87,BonettiSenatoVenezia88} we define
instead the notion of ``super semistandard Young tableaux'' (see
Section 2), where the ordering of $L$ is any. In fact, we are able
to show that many results about Young tableaux do not rely on the
condition $L_0 < L_1$.

In Section 3, for a signed alphabet, we develop in full generality
a notion of ``superplactic monoid'' by means of analogues of the
Knuth relations. We obtain as a by-product a generalization of the
Greene's results about his invariants. In Section 4 we recall a
variant of the Robinson-Schensted-Knuth algorithms establishing a
one-to-one correspondence between pairs of super semistandard
tableaux and two-rowed arrays with entries in signed alphabets
(\cite{BonettiSenatoVenezia88}). A different correspondence for
$(k,l)$-semistandard tableaux also appeared in
\cite{GouldenGreene93}. We give here a new account of the
algorithms in \cite{BonettiSenatoVenezia88} providing an
``implementation-ready'' description of them. We conclude with a
super-analogue of the ``symmetry theorem'' \cite{Schutz63}, with
the aid of some further assumptions.

One application of the RSK-correspondence consists in providing the linear
independence in the celebrated ``standard basis theorem''.
Note that by means of the super RSK we obtain the same result for
the ``super standard basis'' introduced in \cite{GrosshansRotaStein87}.
Let us explain this in more detail.
Let $X$ be a signed alphabet and $F$ a field of characteristic zero.
We denote by $\FX$ the free associative algebra generated by $X$
that is the tensor algebra of the vector space $F X$.
This algebra is $\Z_2$-graded if we put $|w| := |x_1| + \cdots + |x_n|$
for any monomial $w = x_1 \cdots x_n$. Note that a $\Z_2$-graded
algebra is called also a {\em superalgebra}. Denote by $I$ the two-sided
ideal of $\FX$ generated the binomials
$
x_i x_j - (-1)^{|x_i||x_j|} x_j x_i,
$
where $x_i,x_j\in X$. Clearly $I$ is $\Z_2$-graded ideal and we define
$
\Super[X] := \FX/I.
$
This $\Z_2$-graded algebra is called the {\em free supercommutative
algebra generated by $X$}. The identities of $\Super[X]$ defined by
the above binomials are said in fact {\em supercommutative identities}.
It is plain that an $F$-basis of $\Super[X]$ is given by the cosets of the
monomials $w = x_1 \ldots x_n$ such that
$
x_i \leq x_{i+1},\ \mbox{with}\ x_i = x_{i+1}\ \mbox{only if}\ |x_i| = 0.
$
Note that the algebra $\Super[X]$ is isomorphic to the tensor product
$F[X_0] \otimes_F E[X_1]$ where $F[X_0]$ is the polynomial ring in the commuting
variables $x_i\in X_0$ and $E[X_1]$ is the exterior (or Grassmann) algebra
of the vector space $F X_1$.

Consider now $L,P$ two signed alphabets and define for the set $L \times P$
a structure of signed alphabet by ordering in the right lexicographic way
and putting $|(a,b)| = |a| + |b|$, for any $(a,b)\in L \times P$. We put
$\Super[L|P] := \Super[L \times P]$ and, according to \cite{RotaSturmfels88},
we call this supercommutative algebra the {\em letter-place superalgebra}.
A variable $(a,b)\in L \times P$ will be written as $(a|b)$ where
$a$ is said  a ``letter'' and $b$  a ``place''. This comes from the possibility
to embed the tensor superalgebra $\FL$ into $\Super[L \times P]$ simply
by putting
$
w = a_1 \ldots a_n \mapsto m = (a_1|1) \ldots (a_n|n).
$
Note that for the purposes of invariant theory, we do not have to assume
$\char(F) = 0$. Therefore, in characteristic free, a more general
definition of signed alphabet and $\Super[L|P]$ is given which
involves the notion of divided powers (see \cite{GrosshansRotaStein87}).
In this case, $\Super[L|P]$ is known as the {\em four-fold algebra}.

All the monomials $m = (a_1|b_1) \ldots (a_n|b_n)\in \Super[L|P]$
can be written so that
$
(a_i|b_i) \leq (a_{i+1}|b_{i+1}),\ \mbox{with}\ (a_i|b_i) = (a_{i+1}|b_{i+1})\
\mbox{only if}\ |(a_i|b_i)| = 0.
$
Therefore, we can denote them in a combinatorial way as two-rowed arrays
with signed entries that is as
$
S =
\left[
\begin{array}{@{\hskip 2pt}c@{\hskip 3pt}c@{\hskip 3pt}c@{\hskip 2pt}}
a_1 & \ldots & a_n \\
b_1 & \ldots & b_n \\
\end{array}
\right].
$
An important problem for the letter-place algebra consists in describing
an $F$-linear basis given by polynomials that have some invariant behavior.
It was showed in \cite{GrosshansRotaStein87} that we can attach to each
pair $(T,U)$ of super semistandard Young tableaux a suitable invariant
polynomial of the algebra $\Super[L|P]$ which is usually denoted
as $(T|U)$. Up to a sign, $(T|U)$ is a product of polynomials defined
for each row of the tableaux $T,U$. Let $a_1,\ldots,a_k$ and $b_1,\ldots,b_k$
denote the elements of some row of $T$ and $U$ respectively.
In the non-signed case, that is for $L = L_0,P = P_0$, by considering
the matrix whose entries are the commuting variables $(a|b)$ we have that
the row polynomial $(a_1,\ldots,a_k|b_1,\ldots,b_k)$ is defined simply
as the determinant of the $k\times k$ submatrix extracted by the rows
$a_1,\ldots,a_k$ and the columns $b_1,\ldots,b_k$ (see
\cite{GrosshansRotaStein87} for the general definition in the signed case).

By the ``straightening law'' it can be proved that the polynomials
$(T|U)$ span the letter-place superalgebra. Since the super RSK-correspondence
implies a bijection between the elements of the monomial basis
of $\Super[L|P]$ and the pairs of tableaux $(T,U)$, we conclude that
the invariants $(T|U)$ are actually an $F$-linear basis of $\Super[L|P]$.
This result is known as ``super standard basis theorem'' and it can
be proved also with different techniques. Note finally that this theorem
is not only a fundamental tool for invariant theory but it can been
applied as well to the representation theory of general Lie superalgebras
(see \cite{BriniRegonatiTeolis03}).

\section{Super semistandard Young tableaux. Bumping}

Let $\N$ denote the set of positive integers and let $n\in \N$.
A {\em partition} of $n$ is a sequence of integers $\la = (\la_1,\dots,\la_m)$
such that $\la_1 \geq \dots \geq \la_m > 0$ and $\sum \la_i = n$.
The integer $m \in \N$ is called {\em number of parts} or {\em height}
of the partition. We denote $\la\vdash n$ if $\la$ is a partition of $n$.

The {\em Ferrers-Young diagram} of a partition $\la = (\la_1,\dots,\la_m)$
is defined as the set:
\[
D(\lambda):= \{ (i,j)\ |\ 1\leq i\leq m,1\leq j\leq \lambda_i \}.
\]
We can visualize $D(\lambda)$ by drawing a box for each pair $(i,j)$.
For instance, the diagram of $\la = (7,7,5,3,3,1)$ is:
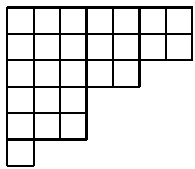
\begin{center}
\begin{picture}(80,70)
\multiput(0,60)(0,-10){3}{\line(1,0){70}} %
\put(0,30){\line(1,0){50}}%
\multiput(0,20)(0,-10){2}{\line(1,0){30}}%
\put(0,0){\line(1,0){10}}%
\multiput(0,0)(10,0){2}{\line(0,1){60}}%
\multiput(20,10)(10,0){2}{\line(0,1){50}}%
\multiput(40,30)(10,0){2}{\line(0,1){30}}%
\multiput(60,40)(10,0){2}{\line(0,1){20}}%
\end{picture}
\end{center}
Note that the transposed diagram $\{(j,i)\mid (i,j)\in D(\la)\}$
defines another partition $\tila\vdash n$ whose parts are the lengths
of the columns of $D(\la)$. The partition $\tila$ is called the {\em conjugate
partition} of $\la$. Finally, let $\la,\mu$ be two partitions. We write
$\mu \subset \la$ if $\la_i \leq \mu_i$ for all $i$. In this case we put
$
D(\la/\mu) := \{ (i,j)\ |\ 1\leq i\leq m,\mu_i < j\leq \lambda_i \}.
$

Although Ferrers-Young diagrams are interesting combinatorial objects on
their own right, we are mainly interested in Young tableaux.
Denote by $\Z_2 = \{0,1\}$ the additive cyclic group of order 2.

\begin{df}
Let $L$ be a finite or countable set. Assume $L$ totally ordered and
let $\abs{\cdot}:L\to \Z_2$ be any map. We call the ordered pair $(L,\abs{\cdot})$
a {\em signed alphabet} and we put $L_0:= \{a\in L\mid |a| = 0\}$
and $L_1:= \{a\in L\mid |a| = 1\}$.
\end{df}

\begin{df}
\label{df: super semistandard Young tableau}
Let $\la\vdash n$ be a partition and let $L$ be a signed alphabet.
A {\em super semistandard Young tableau} is a pair $T := (\la,\tau)$
where $\tau:D(\la)\to L$ is a map such that:
\begin{itemize}
\item[(i)]  $\tau(i,j)\leq \tau(i,j+1)$, with $\tau(i,j)=\tau(i,j+1)$
only if $|\tau(i,j)| = 0$,
\item[(ii)] $\tau(i,j)\leq \tau(i+1,j)$, with $\tau(i,j)=\tau(i+1,j)$
only if $|\tau(i,j)| = 1$.
\end{itemize}
We say that $D(\la)$ is the {\em frame} and $\tau$ the {\em filling}
of the tableau $T$. Moreover, $\lambda$ is called the {\em shape} of $T$.
We say that $T$ is a {\em standard Young tableau} if the map $\tau$
is injective.
\end{df}

In the same way, we can define also a tableau with frame $D(\la/\mu)$
when $\mu \subset \la$. In this case we say that this tableaux has
{\em skew shape} $\la/\mu$. Note that the classic notion of semistandard
Young tableau is recovered when the map $\abs{\cdot}$ has a constant value.
More precisely, if $L = L_0$ we obtain {\em row-strict semistandard tableaux}
and {\em column-strict} ones when $L = L_1$. If we assume the ordering on $L$
is such that $L_0 < L_1$ we obtain the notion of {\em $(k,l)$-semistandard
tableau} introduced in \cite{BereleRemmel85,BereleRegev87}. We do not choose
to make this assumption because many results about tableaux on a signed
alphabet do not depend on it. From now on, unless explicitly stated, we use
the word ``alphabet'' for ``signed alphabet'' and ``Young tableau'' meaning
``super semistandard Young tableau''.

Assume that a Young tableau has been given, and a letter from $L$ as well.
A basic problem in algebraic combinatorics consists in finding a method
for constructing a tableau that includes the new letter and yet is a Young tableau.
In the non-signed case, the Schensted's construction solves the problem in a fully
satisfying way. It can be realized by two ``dual'' algorithms: the insertion
by rows or columns. Here we give generalizations of Schensted's algorithms
that take into account the fact that we are dealing with signed letters. We start
defining the {\em row-insertion}.

\begin{df}
\label{df: row-insertion}
Let $T = (\la,\tau)$ be a Young tableau with $\la = (\la_1,\ldots,\la_m)$
and let $x\in L$. By $T\leftarrow x$ we mean the following procedure:

\vskip 2pt \noindent
\begin{tabular}{@{\hskip 0pt}l@{\hskip 4pt}l}
{\it Step 0}  & Put $\la_{m+1}:= 0$ and $i:= 1$. \\
{\it Step 1a} & If $|x| = 0$ then put $J:= \{1\leq j\leq \la_i\ |\ x < \tau(i,j)\}$. \\
{\it Step 1b} & If $|x| = 1$ then put $J:= \{1\leq j\leq \la_i\ |\ x \leq \tau(i,j)\}$. \\
{\it Step 2}  & If $J = \emptyset$ then go to Step 4. \\
{\it Step 3}  & Put $j:= \min(J)$ and $y:= \tau(i,j)$. Put $\tau(i,j):= x, x:= y,i:= i + 1$ \\
              & and go to Step 1. \\
{\it Step 4}  & Put $\la_i:= \la_i + 1,\tau(i,\la_i):= x$. Output $(\la,\tau)$ and $i$. \\
\end{tabular}

\vskip 2pt \noindent
In the Step 3, we say the letter $x$ {\em bumps} the entry $y$ from the $i$-th row.
We denote the Young tableau obtained by means of this procedure as $[T\leftarrow x]$.
\end{df}

\begin{ex}
Let $L := \N$, with signature given by $L_0 := \{ \mbox{odd numbers} \}$
and $L_1$ defined consequently. Consider the following tableau:
\[
T\ :=\
\begin{matrix}
1 & 1 & 1 & 2 & 4 & 5\\
2 & 3 & 3 & 4\\
2 & 4\\
2 & 4\\
2\\
3
\end{matrix}
\]
and let us row-insert at first the letter $6$. By Step 2, since 6 is greater than
any entry of the first row, we go immediately to Step 4 and hence $[T\leftarrow 6]$
is:
\[
\begin{matrix}
1 & 1 & 1 & 2 & 4 & 5 & 6\\
2 & 3 & 3 & 4\\
2 & 4\\
2 & 4\\
2\\
3
\end{matrix}
\]
Now, let us insert the letter $1$ in the tableau $T':= [T\leftarrow 6]$. We go through
Step 3 and this letter bumps the entry $2$ which now we are trying to place in the
second row. This can be done bumping the already-placed $2$, and then we test its
insertion in the third row. There, and in the rows below, the letter $2$ continues to be
displaced, until the last row, where we test the insertion of $2$. Such insertion can
be done by bumping the letter $3$ and finally by Step 4, this letter forms a new row.
Hence the tableau $[T'\leftarrow 1]$ is the following:
\[
\begin{matrix}
1 & 1 & 1 & 1 & 4 & 5 &6\\
2 & 3 & 3 & 4\\
2 & 4\\
2 & 4\\
2\\
2\\
3
\end{matrix}
\]
\end{ex}

\smallskip
Note that procedure of row-insertion $T\leftarrow x$ can be reversed since we record
the row of the new box added to $T$. In fact, by the definition of row-insertion
such box is added to $T$ as the last in its row and column. Precisely, a procedure
of {\em row-deletion} is defined as follows:

\begin{df}
\label{df: row-deletion}
Let $T = (\la,\tau)$ be a Young tableau with $\la = (\la_1,\ldots,\la_m)$ and
let $1 \leq i\leq m$. We denote by $i\leftarrow T$ the following procedure:

\vskip 2pt \noindent
\begin{tabular}{@{\hskip 0pt}l@{\hskip 4pt}l}
{\it Step 0}  & Put $x:= \tau(i,\la_i)$ and $\la_i:= \la_i - 1$. Put $\la_0:= 0$ and $h:= i - 1$. \\
{\it Step 1a} & If $|x| = 0$ then put $J:= \{1\leq j\leq \la_i\ |\ \tau(i,j) < x\}$. \\
{\it Step 1b} & If $|x| = 1$ then put $J:= \{1\leq j\leq \la_i\ |\ \tau(i,j) \leq x\}$. \\
{\it Step 2}  & If $J = \emptyset$ then go to Step 4. \\
{\it Step 3}  & Put $k:= \max(J)$ and $y:= \tau(h,k)$. Put $\tau(h,k):= x, x:= y,h:= h - 1$ \\
              & and go to Step 1. \\
{\it Step 4}  & Output $(\la,\tau)$ and $x$. \\
\end{tabular}

\vskip 2pt \noindent
We write $[i\leftarrow T]$ for the tableau we obtain.
\end{df}

Resembling the Definition \ref{df: row-insertion}, a column-insertion algorithm
can be also defined.

\begin{df}\label{df: column-insertion}
Let $T = (\la,\tau)$ be a tableau and let $x\in L$. Denote by $\tila = (\tila_1,\ldots,\tila_r)$
the conjugate partition of $\la$. We define $x\rightarrow T$ the following procedure:

\vskip 2pt \noindent
\begin{tabular}{@{\hskip 0pt}l@{\hskip 4pt}l}
{\it Step 0}  & Put $\tila_{r+1}:= 0$ and $j:= 1$. \\
{\it Step 1a} & If $|x| = 0$ then put $I:= \{1\leq i\leq \tila_j\ |\ x \leq \tau(i,j)\}$. \\
{\it Step 1b} & If $|x| = 1$ then put $I:= \{1\leq i\leq \tila_j\ |\ x < \tau(i,j)\}$. \\
{\it Step 2}  & If $I = \emptyset$ then go to Step 4. \\
{\it Step 3}  & Put $k:= \min(I)$ and $y:= \tau(i,j)$. Put $\tau(i,j):= x, x:= y,j:= j + 1$ \\
              & and go to Step 1. \\
{\it Step 4}  & Put $\tila_j:= \tila_j + 1,\tau(\tila_j,j):= x$. Output $(\la,\tau)$ and $j$. \\
\end{tabular}

\vskip 2pt \noindent
The tableau obtained by means of this procedure will be denoted as $[x\rightarrow T]$.
\end{df}

The definition of the {\em column-deletion} is left to the reader. Such procedure
will be denoted as $T\rightarrow i$.

\smallskip
\begin{rem}
\label{rem: row-column duality}
The two procedures of Definitions \ref{df: row-insertion} and \ref{df: column-insertion}
are related by a simple duality involving the signature on $L$.
Namely, let $L$ be an alphabet and define $\tilde{L}$ by putting $\tilde{L}_0 := L_1$
and $\tilde{L}_1 := L_0$. We call $\tilde{L}$ the {\em conjugate alphabet of $L$}.
If $T = (\la,\tau)$ is a Young tableau, denote by $\tilde{T}$ the pair $(\tila,\titau)$
where $\tila$ is the conjugate partition of $\la$ and $\titau(i,j) := \tau(j,i)$
for any $(i,j)\in D(\tila)$. Clearly $\tilde{T}$ is a Young tableau for the conjugate
alphabet $\tilde{L}$. Then, from the definitions it follows easily that:
\[
[x\rightarrow T] = \widetilde{[\tilde{T}\leftarrow x]}
\]
where in the right hand side the letter $x$ has to be considered as an element
of the conjugate alphabet.
\end{rem}

\begin{prop}
\label{prop: correctness of insertion}
Let $T$ be a Young tableau and $x\in L$. Then both $[T\leftarrow x]$ and
$[x\rightarrow T]$ are Young tableaux.
\end{prop}

\begin{proof}
Put $T' := [T\leftarrow x]$. Clearly, there is no loss of generality on assuming
that $T'$ has just two rows, say $w'_1,w'_2$. Denote by $w_1,w_2$ the rows
of $T$ with $w_2$ possibly empty. By Definition \ref{df: row-insertion}
of the row-insertion process, it is plain that $w'_1,w'_2$ verify condition (i)
of Definition \ref{df: super semistandard Young tableau}. Then, we have to check
the shape and the columns of $T'$. If $x$ is appended at the end of $w_1$ there
is nothing to prove. Otherwise, define $y$ the entry of $w_1$ bumped by $x$ and
denote by $j,j'$ the columns where $y$ occurs in the rows $w_1,w'_2$ respectively.
We claim that $j' \leq j$. This implies that if $y$ is appended at the end
of the row $w_2$ then the length of $w_2$ is strictly less than the one of $w_1$.
Moreover, we have that $x \leq y$, with $x = y$ only if $|x| = 1$ by Definition
\ref{df: row-insertion} and $z \leq x$ for any entry $z$ of $w'_1$ at the right
or equal to $x$. By $j' \leq j$ this implies that the entry above $y$ in the
row $w'_1$ satisfies condition (ii) of Definition \ref{df: super semistandard
Young tableau}. By contradiction, assume now that $j < j'$ and let $z$ be the
entry of $w'_2$ below $x$. Since $T$ is a Young tableau and $z$ is below $y$
in $T$ we have that $y \leq z$, with $y = z$ only if $|y| = 1$. But $z$ is at
the left of $y$ in $w'_2$ and therefore $y \geq z$, with $y = z$ only if $|y| = 0$
which is impossible.

By Remark \ref{rem: row-column duality} it follows immediately that
$[x\rightarrow T]$ is also a Young tableau.
\end{proof}

Let $T',i$ be the tableau and the row-index defined as the output of the procedure
$T \leftarrow x$ and put $j$ the length of the row $i$ in $T'$. In other words,
$(i,j)$ is the position of the new box added to the shape of $T$ for obtaining
the shape of $T'$. Consider also $T'',i'$ the output of $T' \leftarrow x'$ and
let $j'$ be the length of the row $i'$ in $T''$. We have the following result.

\begin{prop}[Row-bumping lemma]
\label{prop: row-bumping lemma}
The following conditions are equivalent:
\begin{itemize}
\item[(a)] $x \leq x'$, with $x = x'$ only if $|x| = 0$,
\item[(b)] $j < j'$ (and hence $i \geq i'$).
\end{itemize}
\end{prop}

\begin{proof}
For any $q$ ($1\leq q \leq i$) we define an entry $x_q$ of the $q$-th
row of $T'$ as follows:
\begin{itemize}
\item[(i)]  $x_1:= x$,
\item[(ii)] $x_{q+1}$ is the entry of the row $q$ bumped by the
element $x_q$.
\end{itemize}
We call $x_1,x_2,\ldots,x_i$ the {\em bumping sequence} defined by
the row-insertion $T \leftarrow x$. By Definition
\ref{df: row-insertion} we have that
$
x_q \leq x_{q+1},\ \mbox{with}\ x_q = x_{q+1}\ \mbox{only if}\ |x_q| = 1.
$
Denote by $j_q$ the column where the entry $x_q$ occurs in the row $q$
of $T'$. As proved in the argument of Proposition \ref{prop: correctness
of insertion} we have that $j_q \geq j_{q+1}$.

For the procedure $T' \leftarrow x'$ we define in the same way the bumping
sequence $x'_1,x'_2,\ldots,x'_{i'}$ and denote by $j'_q$ the column where
the entry $x'_q$ is placed on the $q$-th row of $T''$.
The basic argument of the proof consists in observing that the following
statements are equivalent:
\begin{itemize}
\item[($\alpha$)] $x_q \leq x'_q$, with $x_q = x'_q$ only if $|x_q| = 0$,
\item[($\beta$)]  $j_q < j'_q$.
\end{itemize}
In fact, by Definition \ref{df: row-insertion} of row-insertion from ($\alpha$)
it follows that ($\beta$) holds. Moreover, since $T''$ is a Young tableau
we have that ($\beta$) implies ($\alpha$).

Let us assume the condition (a) holds. Since $x \leq x'$, with $x = x'$ only if
$|x| = 0$ we have that $j_1 < j'_1$. By induction on the row-index $q$,
we prove that $j_q < j'_q$ for any $q$ ($1\leq q\leq \min(i,i')$). In fact,
it is sufficient to note that before we insert the elements $x_q,x'_q$
in the row $q$ of the tableau $T$, we have the elements $x_{q+1},x'_{q+1}$
in those positions. By assuming $j_q < j'_q$ we have therefore
$x_{q+1} \leq x'_{q+1}$, where $x_{q+1} = x'_{q+1}$ implies that $|x_{q+1}| = 0$.
In the row $q+1$ where the letters $x_{q+1},x'_{q+1}$ are inserted,
this means that $j_{q+1} < j'_{q+1}$.

If we show that $i' \leq i$ we can conclude that
$
j' = j'_{i'} > j_{i'} \geq j_i = j.
$
Note that the entry $x_i$ is placed at the end of the $i$-th row of the
tableau $T'$. The same happens to the entry $x'_{i'}$ of $T''$. Moreover,
the shapes of $T',T''$ differ only for the box of position $(i',j')$.
If we suppose that $i \leq i'$, by $j = j_i < j'_i$ we have that
$i' = i$ and $j' = j'_i = j + 1$. This implies that $i' \leq i$.

\smallskip
We assume now (b). From $j < j'$ it follows that $i \geq i'$.
Since the entry $x'_{i'}$ is the last one of the row $i'$ of $T''$
we have $j_{i'} < j'_{i'} = j'$. By backward induction on $q$
($i' \geq q \geq 1$) it is proved that $j_q < j'_q$.
In particular, one has $j_1 < j'_1$ and hence $x \leq x'$, with $x = x'$
only if $|x| = 0$.
\end{proof}

Let $T',j$ be the tableau and the column-index that are the output of the
procedure $x \rightarrow T$ and put $i$ the length of the column $j$ in $T'$.
Moreover, denote by $T'',j'$ the output of $x' \rightarrow T'$ and let $i'$
be the length of the row $j'$ in $T''$. A ``dual result'' of Proposition
\ref{prop: row-bumping lemma} can be obtained immediately by Remark
\ref{rem: row-column duality}.

\begin{prop}[Column-bumping lemma]
\label{prop: column-bumping lemma}
We have the following equivalent conditions:
\begin{itemize}
\item[(a)] $x \leq x'$, with $x = x'$ only if $|x| = 1$,
\item[(b)] $i < i'$ (and hence $j \geq j'$).
\end{itemize}
\end{prop}

\section{The super plactic monoid}

While Schensted's construction \cite{Schensted61} was originally made in order
to count some invariants of sequences of symbols, it was Knuth \cite{Knuth70}
who refined the insertion algorithm revealing the inner structure of it. Then,
there were Lascoux and Sch\"utzen\-berger \cite{LascouxSchutz81} who connected
it to a relevant algebraic structure they called the {\em plactic monoid}. Here,
we are going to generalize that structure to the case of a signed alphabet.

Let $L$ be an alphabet and denote by $\L$ the free monoid generated by $L$.
It consists of all words $w := x_1 \cdots x_n$ ($x_i\in L$)
endowed with juxtaposition product. We admit the empty word belonging to $\L$
and behaving as the unity for the product.

\begin{df}
A monoid $\mc{M}$ is said a {\em $\Z_2$-graded monoid} or a {\em supermonoid}
if a map $\abs{\cdot}:\mc{M} \to \Z_2$ is given such that $|u\cdot v| =
|u| + |v|$, for any $u,v\in\mc{M}$. We call $|u|$ the {\em $\Z_2$-degree}
of the element $u$.
\end{df}

If $L$ is a signed alphabet, note that the free monoid $\L$ is $\Z_2$-graded
simply by putting $|w|:= |x_1| + \ldots + |x_n|$, for any word $w =
x_1 \cdots x_n$.

Any Young tableau $T$ defines a word $w(T)$ of $\L$ in a simple way:
start collecting the rows of $T$ from its bottom upward and multiply those
sequences of letters by juxtaposition. For instance, if we consider the tableau:
\[
T := \begin{matrix}
1 & 1 & 1 & 1 & 4 & 5 \\
2 & 3 & 3 & 4 \\
2 & 4 \\
2 & 4 \\
2 \\
2 \\
3
\end{matrix}
\]
then it defines the word:
\[
w(T) := 3\,2\,2\,2\,4\,2\,4\,2\,3\,3\,4\,1\,1\,1\,1\,4\,5.
\]
Moreover, any word $w = x_1 \dots x_n$ defines a Young tableau $T(w)$
by inserting its letters progressively starting from the empty tableau.
Choosing the insertion by rows, we define:
\[
T(w) := [[[\varnothing\leftarrow x_1]\leftarrow x_2]\dots \leftarrow x_n]
\]
where $\varnothing$ denotes the empty tableau.

\begin{prop}
\label{step0}
Let $L$ be a signed alphabet. For any Young tableau $T$, it holds:
\[
T(w(T)) = T
\]
\end{prop}

\begin{proof}
Denote by $w_i$ the $i$-th row of $T$ and let $T'$ be the tableau obtained
by considering all the rows of $T$ but $w_1$. By induction on the number
of rows we can assume that $T(w(T')) = T'$. Note that $w(T) = w(T') w_1$.
Hence, if $w_1 = x_1 \cdots x_n$ we have:
\[
T(w(T)) = [[[T'\leftarrow x_1]\leftarrow x_2]\dots \leftarrow x_n].
\]
We claim that the bumping sequence defined by the letter $x_i$ is all
in the $i$-th column. This clearly implies that $T(w(T)) = T$.
Let $T_i := [[[T'\leftarrow x_1]\leftarrow x_2]\dots \leftarrow x_i]$
and say $w_2 = y_1 \dots y_m$ ($m \leq n$). By induction on $i$,
the first row of the tableau $T_i$ is $x_1 \cdots x_i y_{i+1} \cdots y_m$
if $i < m$ or $x_1 \cdots x_i$ otherwise. By applying the row-insertion
$T_i\leftarrow x_{i+1}$ we have that $x_{i+1}$ bumps $y_{i+1}$ if
$i < m$ or $x_{i+1}$ is appended at the end of the first row.
In fact, since $T$ is a Young tableau we have that $x_i \leq x_{i+1}$,
with $x_i = x_{i+1}$ only if $|x_i| = 0$ and $x_{i+1} \leq y_{i+1}$,
where $x_{i+1} = y_{i+1}$ only if $|x_{i+1}| = 1$. If $i < m$ then
$y_{i+1}$ has to be inserted in the second row. Similar arguments apply
to this letter that has to be placed in the column $i+1$ and so on
for the remaining bumping sequence.
\end{proof}

It is almost a natural question to wonder when different words
give rise to the same Young tableau. Note that there are words $w$ such that $w(T(w)) \neq w$. We show that any pair of words $w,w(T(w))$ is related by a congruence on the monoid $\L$ which is compatible with its $\Z_2$-grading.

\begin{df}
We say that $w = x_1 \cdots x_n$ is a {\em row word} if we have that
$x_i \leq x_{i+1}$, with $x_i = x_{i+1}$ only if $|x_i| = 0$. In other words,
we have $w = w(T)$ where $T$ is a Young tableau whose shape is a row.
In a similar way we can define the notion of {\em column word}. Note that
all the words of length $\leq 2$ are clearly row or a column words.
\end{df}

\begin{df}
Denote by $\sim$ the equivalence relation on $\L$ defined by the map
$w \mapsto T(w)$, that is
$
w \sim w'\ \mbox{if and only if}\ T(w) = T(w').
$
\end{df}

Clearly, if $w,w'$ are row or column words then $w \sim w'$ if and only if
$w = w'$. This happens in particular if the length of these words is $\leq 2$.
So, the first non-trivial relations can be found for words of length 3 that
correspond to tableaux of shape $(2,1)$. In fact, a straightforward computation
provides the following:

\begin{lem}
\label{knuthrel}
Let $x \leq y \leq z$ be letters of the alphabet $L$. We have:
\begin{itemize}
\item[(i)]
$
T(x z y) = T(z x y) =
\begin{array}{l}
\begin{array}{|c|c|}
\hline
x & y \\
\hline
\end{array}
\\
\begin{array}{|c|}
\hspace{.6pt} z \\
\hline
\end{array}
\end{array}
$
with $x = y$ only if $|y| = 0$ and $y = z$ only if $|y| = 1$,
\item[(ii)]
$
T(y x z) = T(y z x) =
\begin{array}{l}
\begin{array}{|c|c|}
\hline
x & z \\
\hline
\end{array}
\\
\begin{array}{|c|}
\hspace{.45pt} y \\
\hline
\end{array}
\end{array}
$
with $x = y$ only if $|y| = 1$ and $y = z$ only if $|y| = 0$.
\end{itemize}
\end{lem}

\smallskip
\begin{df}
Denote by $\equiv$ the congruence on the monoid $\L$ generated
by the following relations:
\begin{itemize}
\item[(K1)] $x z y \equiv z x y$, with $x = y$ only if $|y| = 0$
and $y = z$ only if $|y| = 1$,
\item[(K2)] $y x z \equiv y z x$, with $x = y$ only if $|y| = 1$
and $y = z$ only if $|y| = 0$,
\end{itemize}
for any triple $x\leq y\leq z$ of elements of $L$. Note that in the
classic case, that is for $L = L_0$, we have that (K1),(K2) are exactly
the {\em Knuth's relations} \cite{Knuth70}. Then, we put:
\[
\Pl(L) := \L/\equiv
\]
and we call such monoid the {\em super plactic monoid} defined by the signed
alphabet $L$. Since the relations (K1),(K2) are both $\Z_2$-homogeneous we have
that the quotient $\Pl(L)$ is actually a supermonoid.
\end{df}

We are going to show now that the monoid $\Pl(L)$ coincide with the quotient
set $\L/\sim$. This allows to define on the set of super semistandard Young
tableaux a structure of $\Z_2$-graded monoid which can be identified with
$\Pl(L)$ by means of the map $w \mapsto T(w)$. The arguments we use for obtaining
this result follow closely the approach introduced in \cite{Lothaire01}, chapter 5,
for the non-signed case. We start with some preparatory results.

\begin{prop}
\label{step1}
Let $L$ be a signed alphabet. For all words $w\in\L$, we have
$w \equiv w(T(w))$.
\end{prop}

\begin{proof}
We argue by induction on the length of the word $w$. To simplify the notation,
we use here the same letter $T$ for denoting a tableau and its corresponding
word $w(T)$. If the length of $w$ is $\leq 3$ then either $w$ is a row or column
word, or it occurs in the relations (K1),(K2). Hence, by Lemma \ref{knuthrel}
we conclude for this case that $w \equiv T(w)$.

Now, assume that $w \equiv T(w)$ and let $x$ be an element of $L$.
We have to prove that $w x \equiv T(w x)$ that is $T(w) x \equiv T(w x)$.
Note that:
\[
T(w x) = [T(w) \leftarrow x].
\]
Since the procedure $T(w) \leftarrow x$ is defined row-by-row and being
$\equiv$ a congruence, we need only to consider the case when $w$ is
a row and hence $T(w) = w$. Then, we have to prove that $w x$ is congruent
to the word obtained by row-inserting $x$ in $T(w)$. Note immediately
that if $w x$ is a row then $T(w x) = w x$. Otherwise, we have two cases
depending on the value of $|x|$.

If $|x| = 1$ then $T(w x) = y w'$, where $y$ is the leftmost entry of $w$
among those $\geq x$ and $w'$ is obtained from $w$ by replacing $y$ with $x$.
By putting $w = u y v$ we have hence that $w x = u y v x$ and $T(w x) = y u x v$.
Note that if $a,b$ are any letters of the words $u,v$ respectively then $a < x \leq y\leq b$. Moreover, one has $y = b$ only if $|y| = 0$ since
$y v$ is a row. Then, by applying iteratively the congruence (K2) we get:
\[
w x = u y v x \equiv u y x v.
\]
Making use of the congruence (K1) we are able to conclude that:
\[
u y x v \equiv y u x v = T(w x).
\]
In case $|x| = 0$, one has $a < x < y \leqslant b$, and in a similar way
we argue that $wx\approx T(wx)$.

In general, $T(m)$ is a product $w_nw_{n-1}\dots w_2w_1$ of rows $w_i$ and,
since the insertion $[T(m)\leftarrow x]$ is through rows and $\approx$ is a congruence,
one can apply iteratively what has been done for one row and go through the general case.
\end{proof}

Let $v,w$ be words of $\L$ and let $w = x_1\ldots x_n$. We say
that the word $v$ is {\em extracted by $w$} if $v = x_{i_1} \cdots
x_{i_m}$ with $1 \leq i_1 < \ldots < i_m \leq n$.
\begin{df}
For any $k \in\N$, we denote by $l_k(w)$ the maximal number which
can be obtained as the sum of the lengths of $k$ row words that
are disjoint and extracted by $w$. In the same way, we define
$\til_k(w)$ as the maximal value among the sums of the lengths of
$k$ column words disjoint and extracted by $w$.
\end{df}
Clearly, the integers $l_k(w)$ and $\til_k(w)$ are the
super-analogues of {\em Greene's row and column invariants}
\cite{Greene74}.

\smallskip
For instance, if $L_0 = \{1,2,4\},L_1 = \{3,5\}$ and $w = 1 2 3 3 4 5 5$
we have:
\[
\begin{array}{l}
l_1(w) = 5,\ l_2(w) = l_3(w) = \ldots = 7, \\
\til_1(w) = 2,\ \til_2(w) = 4,\ \til_3(w) = 5,\ \til_4(w) = 6,\
\til_5(w) = \til_6(w) = \ldots = 7.
\end{array}
\]

The motivation to call the above numbers ``invariants'' is that they stay
stable over the plactic classes.

\begin{prop}
\label{step2}
If $w \equiv w'$ then $l_k(w) = l_k(w')$ for any $k$.
\end{prop}

\begin{proof}
We can assume that the words $w,w'$ are congruent by means of a single generating relation, say (K1). Then:
\[
w = u x z y v,\ w' = u z x y v,
\]
with $u,v$ words and $x \leq y \leq z$ letters such that $x = y$ only if
$|y| = 0$ and $y = z$ only if $|y| = 1$ (hence $x \neq z$). Note that
two row words can be extracted by the word $x z y$, precisely $x y$ and $x z$.
Moreover, by the word $z x y$ we can extract just $x y$. Then, all the words
extracted by $w'$ can be also extracted by $w$ and hence $l_k(w) \geq l_k(w')$.

Assume now that $(w_1,\ldots,w_k)$ is a $k$-tuple of row words that are
disjoint and extracted by $w$. We want to prove that there is a $k$-tuple
$(w'_1,\ldots,w'_k)$ of row words disjoint and extracted by $w'$ such that
the sum of the lengths of the $w'_i$ is equal to the corresponding sum for the $w_i$.
In this case, in fact, we get clearly $l_k(w) \leq l_k(w')$. Note that the word
$w_i$ is extracted also by $w'$ unless that $w_i = u' x z v'$ with $y \neq z$
and $u',v'$ words extracted by $u,v$. If the letter $y$ does not occur in any
of the words $w_j$ ($i \neq j)$ then we can obtain the $k$-tuple
$(w'_1,\ldots,w'_k)$ from $(w_1,\ldots,w_k)$ simply by substituting the row
word $w_i$ with $w'_i = u' x y v'$. Otherwise, if we have $w_j = u'' y v''$ then
we replace the pair $(w_i,w_j)$ with $w'_i = u' x y v''$ and $w'_j = u'' z v'$.
In a similar way we argue for the case when $w \equiv w'$ by means of (K2).
\end{proof}

\begin{thm}[Greene's theorem]
\label{greene}
Let $L$ be a signed alphabet. Let $w$ be a word and denote by $\lambda =
(\lambda_1,\ldots,\lambda_r)$ the shape of the tableau $T(w)$. Moreover,
consider $\tila = (\tila_1,\ldots,\tila_s)$ the conjugate partition
of $\lambda$. For any $k\in\N$ we have:
\[
l_k(w) = \lambda_1 + \ldots + \lambda_k,\
\til_k(w) = \tila_1 + \ldots + \tila_k.
\]
\end{thm}

\begin{proof}
By using the previous propositions, it is sufficient to prove that
$l_k(w(T)) = \lambda_1 + \ldots + \lambda_k$ for any Young tableau $T$.
If we consider the first $k$ rows of $T$, it is clear that
$l_k(w(T)) \geq \lambda_1 + \ldots + \lambda_k$. Conversely, if $w$ is
a row word of $w(T)$ then $w$ cannot have two letters in the same
column of $T$. Hence, a $k$-tuple of row words disjoint and extracted
by $w(T)$ can be done with at most $\lambda_1 + \ldots + \lambda_k$
letters of $T$. In other words, we have that $l_k(w(T)) \leq
\lambda_1 + \ldots + \lambda_k$.

A similar argument can be used for proving that $\til_k(w) =
\tila_1 + \ldots + \tila_k$.
\end{proof}

An analogous result has been proved in \cite{BereleRemmel85} for
the language of ``$(k,l)$-semistandard tableaux'' that is under
the assumption that $L_0 < L_1$. We can finally state the main
result of this section.

\begin{thm}[Cross-section theorem]
Let $L$ be a signed alphabet. The equivalence relation $\sim$
coincide with the congruence $\equiv$ on the monoid $\L$. Moreover,
the maps $w \mapsto T(w)$ and $T \mapsto w(T)$ define a one-to-one
correspondence between the super plactic monoid $\Pl(L)$ and the set
of super semistandard Young tableaux over $L$. In particular,
a cross-section of the plactic classes is given by the words
$w(T)$, as $T$ ranges over the set of Young tableaux.
\end{thm}

\begin{proof}
If $T(w) = T(w')$ then by Proposition \ref{step1} we have:
\[
w \equiv w(T(w)) = w(T(w')) \equiv w'.
\]
Now assume $w \equiv w'$. By Proposition \ref{step2} and
Theorem \ref{greene} it follows that the tableaux $T(w),T(w')$ have
the same shape. Denote by $z$ the greatest letter of the words
(with same content) $w,w'$ and put:
\[
w = u z v,\ w' = u' z v'
\]
where $u,v,u',v'$ words. Moreover, assume that if $|z| = 0$ then
$z$ does not occur in $v,v'$ and if $|z| = 1$ then $z$ does not
occur in $u,u'$.

We want to prove that $u v \equiv u' v'$. We can assume that $w \equiv w'$
by means of one of the generating relations (K1),(K2). If $z$ does not occur
in such relation then clearly either $u \equiv u'$ and $v = v'$ or $u = u'$
and $v \equiv v'$. Otherwise, note that by deleting the letter $z$ in (K1)
or (K2) we have that $x y = x y$ or $y x = y x$. By the assumptions
on the words $u,v,u',v'$ with respect to $|z|$ we have therefore that
$u v = u' v'$.

By induction on the length of $w$, we can assume hence that $T(u v) = T(u' v')$.
Since $z$ is the greatest among the entries of $T(w) = T(u z v)$ and therefore
it occurs at the end of a row and of a column, it is plain that if we delete $z$ by a suitable
row of $T(w)$ then we get the tableau $T(u v)$. Since $T(u v) = T(u' v')$ and
the shapes of $T(w),T(w')$ are the same, we conclude that $T(w) = T(w')$.

Note finally that Proposition \ref{step0} states that $T(w(T)) = T$
for any tableau T. Now, by Proposition \ref{step1} we have also that
$[w(T(w))] = [w]$, for any plactic class $[w]$.
\end{proof}

Let us assume, in the rest of this section, $L$ being finite. It
is a well known fact that it is possible to build a ring, the so
called {\em group ring}, rising from a monoid. In our case, it is
a \Z-algebra, whose linear generators are the monomials in $Pl(L)$
(or, to be more precise, the plactic classes). This is an
associative and unitary (but not commutative) ring. Let us denote
it by $R_L$. A generic element in it can be realized by a formal
sum of classes with coefficients from $\Z$. Since every class can
be represented by a tableau, a typical element in $R_L$ is a
formal sum of tableaux. If $X:=\{x_a\mid a\in L\}$ is a set of
free variables, the map $a\in L\To x_a$ leads to an homomorphic
image of $R_L$ in the free commutative \Z-algebra $\Z[X]$.

Traditionally, the element
\[
S_{\la}=\sum_{T(w)\text{ has shape } \la}w
\]
(with $\la\vdash n$) gives rise to the so called {\em Schur function} $s_{\la}(X)$.
The properties of $R_L$ reflect themselves in the commutative ring $\Z[X]$ in several ways. A striking result is the Pieri-Young rule:
\begin{thm}\label{thm: classical Pieri-Young}
Let $\la$ be a partition of $n$. Then, for $p\in \N$,
\[
s_{\la}(X)\cdot s_{(p)}=\sum_{\mu}s_{\mu}
\]
where the sum is on all partitions $\mu\vdash(n+p)$ such that
$\la\subseteq \mu$ and the skew-diagram $\mu/\la$ does not
contains two boxes in the same column.

Similarly,
\[
s_{\la}(X)\cdot s_{(1^p)}=\sum_{\mu}s_{\mu}
\]
where the sum is on all partitions $\mu\vdash(n+p)$ such that
$\la\subseteq \mu$ and the skew-diagram $\mu/\la$ does not
contains two boxes in the same row.
\end{thm}

Now we are going to prove a similar result within our settings. So, let $L$ be a
proper finite signed alphabet, and let us denote by $R_L$ the group ring
associated to the superplactic monoid $\mc{M}(L)/\approx$.

As a consequence of Proposition \ref{prop: row-bumping lemma}, we get

\begin{cor}\label{cor: factorization by a row}
Let $T$ be a tableau, $\la\vdash n$ be its frame, and let $w$ be a
row word of length $p$. Then $[T\leftarrow w]$ is a tableau of
frame $\mu\vdash (n+p)$ containing $\la$ and such that the skew
diagram $\mu/\la$ does not contain any pair of boxes in the same
column.

Conversely, if $U$ is a tableau of frame $\mu\vdash (n+p)$, and
$\la$ is a diagram contained in $\mu$ and such that in the skew
diagram $\mu/\la$ no two boxes occur in the same column, then
there exist a unique tableau $T$ with frame $\la$ and a row word
$w$ of length $p$ such that $U=[T\leftarrow w]$.
\end{cor}
\begin{proof}
Let $T$ be a tableau with shape $\la$ and let $w=x_1\dots
x_p$. Let $(r_i,c_i)$ be the final extra box for the insertion  $(T\leftarrow x_1\dots x_{i-1})\leftarrow x_i$ in the tableau $[T\leftarrow
x_1\dots x_{i-1}]$ for each $i\leqslant p$. Notice that $\mu/\la$ is the skew diagram
consisting of the boxes $(r_1,c_1),\dots,(r_p, c_p)$. Moreover, since $w$ is a row word, it holds $x_1\leqslant
x_2\leqslant\dots \leqslant x_p$ and, if  $x_i=x_{i+1}$, then $\abs{x_i}=0$. Therefore, by Proposition \ref{prop: row-bumping lemma},
for all $i\leqslant p-1$ the boxes $(r_i,c_i)$ e $(r_{i+1},c_{i+1})$ satisfy
$c_i<c_{i+1}$. Hence no two boxes of the skew-diagram
$\mu/\la$ are in the same column.

Conversely, if $T_1$ is a tableau with shape $\mu$ and such that no two boxes of $\mu/\la$ are in the same column, then there exists a unique row word $w$ such that $[T\leftarrow w]=T_1$. The reason is the following: starting from the rightmost box of $\mu/\la$, say $(r_p,c_p)$, we may perform the row estraction algorithm, as in Definition \ref{df: row-deletion}, and we get a letter $x_p$ and a tableau with shape $\mu\setminus \{(r_p,c_p)\}$. Then we repeat the algorithm, at each step moving leftward. The word $w=x_1\dots x_p$ obtained this way is a row word and, of course, $[T\leftarrow w]=T_1$.
The uniqueness is clear: if $w'$ a word with the same property of $w$ then $w$ and $w'$
must have the same content, because $[T\leftarrow w]=[T\leftarrow w']$. Then, since
$w\approx w'$ and both must be row words, one has $w=w'$.
\end{proof}

With other words, this means that
\begin{thm}
Let $\la\vdash n$ and let $S_{\la}$ be the (formal) sum of all tableaux of shape
$\la$ over the proper signed alphabet $L$. Let $S_{(p)}$ be the sum of all row words
of length $p$. Then
\[
S_{\la}\cdot S_{(p)}=\sum_{\mu}S_{\mu}
\]
where $\mu$ runs over all partitions of $n+p$ such that
$\la\subseteq \mu$ and no two boxes of $\mu/\la$ occur in the same
column.
\end{thm}
\begin{proof}
It holds
\[
S_{\la}\cdot S_{(p)}=\sum T_{\la}T_{(p)}.
\]
By applying Corollary \ref{cor: factorization by a row},
each tableau $T_{\la}T_{(p)}$ is of type $T_{\mu}$ with $\mu$
partition of $n+p$ containing $\la$ and satisfying the condition that no two boxes of $\mu/\la$ are in the same column. On the other hand, any tableau $T_{\mu}$ with these properties can be factorized uniquely as $T_{\la}T_{(p)}$. This proves the statement.
\end{proof}

Similar considerations provide a result extending the Pieri-Young
formula concerning the column words.

\section{Super RSK-correspondence}

If we want to count the number of words in each plactic class, we have
to transform the map $w \mapsto T(w)$ into a bijective correspondence
between words and pairs of Young tableaux. In each pair, one of the
tableaux has to be standard since related to the positions of the letters
in the corresponding word. If we consider instead any pair of Young
tableaux, the correspondence is with the so-called ``two-rowed arrays''.
We explain here these phenomena within the setting of signed alphabets.

\smallskip
Let $(L,\abs{\cdot}')$ and $(P,\abs{\cdot}'')$ be two alphabets.
On the product set $L \times P$ we can define a structure of alphabet
as follows. The total order is right lexicographic, that is:
\[
(a_1,b_1) < (a_2,b_2),\ \mbox{if}\ b_1 < b_2\ \mbox{or}\ b_1 = b_2, a_1 < a_2.
\]
Moreover, the map $\abs{\cdot}:L \times P \to \Z_2$ is defined as
$|(a,b)| := |a|' + |b|''$. To simplify the notation, from now on we will
denote by $\abs{\cdot}$ all the mappings from the sets $L,P$ and
$L\times P$ to $\Z_2$.

\begin{df}
\label{df: signed two-rowed array}
Let $L,P$ be two alphabets and let
$
S :=
\left[
\begin{array}{@{\hskip 2pt}c@{\hskip 3pt}c@{\hskip 3pt}c@{\hskip 2pt}}
a_1 & \ldots & a_n \\
b_1 & \ldots & b_n \\
\end{array}
\right]
$
be a $2 \times n$ matrix with $a_i\in L,b_i\in P$. We call $S$ a {\em signed
two-rowed array} on the alphabets $L,P$ if its entries satisfies the following
condition:
\[
(a_i,b_i) \leq (a_{i+1},b_{i+1}),\ \mbox{with}\
(a_i,b_i) = (a_{i+1},b_{i+1})\ \mbox{only if}\
|(a_i,b_i)| = 0.
\]
\end{df}

In Section 1 we have explained how to associate these combinatorial
objects to the monomials of the letter-place superalgebra.
They are related to the Young tableaux by means of the following result:

\begin{thm}\label{correspondence}[Super RSK-correspondence]
Let $L,P$ be signed alphabets. A one-to-one correspondence is given between
signed two-rowed arrays and pairs of super semistandard Young tableaux on $L,P$.
Precisely, if we denote by $S \mapsto (T,U)$ this mapping, we have that $T,U$
are tableaux of the same shape whose entries are the ones of the first
and the second row of $S$ respectively.
\end{thm}

This theorem is due to F. Bonetti, D. Senato and A. Venezia
\cite{BonettiSenatoVenezia88} and it is based on a variant of the
Robinson-Schensted-Knuth algorithms. They were originally
presented in the language of the four-fold algebra, and here we
give an equivalent formulation in terms of signed two-rowed
arrays. From now on, unless explicitly stated, we use the word
``two-rowed array'' for ``signed two-rowed array''.

\begin{df}
\label{df: SmapstoTU}
Let
$
S = \left[
\begin{array}{@{\hskip 2pt}c@{\hskip 3pt}c@{\hskip 3pt}c@{\hskip 2pt}}
a_1 & \ldots & a_n \\
b_1 & \ldots & b_n \\
\end{array}
\right]
$
be a two-rowed array on the alphabets $L,P$. The map $S \mapsto (T,U)$ is defined
by the following algorithm:

\smallskip \noindent
\begin{tabular}{@{\hskip 0pt}l@{\hskip 4pt}l}
{\it Step 0}  & Put $T := \varnothing,U := \varnothing$ and $k := 1$. \\
{\it Step 1}  & Put $x := a_k$ and $y := b_k$. \\
{\it Step 2a} & If $|y| = 0$ then put $T,i$ the output of $T \leftarrow x$
                and append y at the end \\
              & of row $i$ of $U$. \\
{\it Step 2b} & If $|y| = 1$ then put $T,j$ the output of $x \rightarrow T$
                and append y at the end \\
              & of column $j$ of $U$. \\
{\it Step 3}  & Put $k := k + 1$. \\
{\it Step 4}  & If $k \leq n $ then go to Step 1 else output $T,U$. \\
\end{tabular}
\end{df}

\smallskip
\begin{df}
\label{df: TUmapstoS}
Let $T,U$ a pair of Young tableaux of the same shape on the alphabets
respectively $L,P$. The map $(T,U) \mapsto S$ is by definition
the following algorithm:

\smallskip \noindent
\begin{tabular}{@{\hskip 0pt}l@{\hskip 4pt}l}
{\it Step 0}  & Put
$
S :=
\left[
\begin{array}{@{\hskip 2pt}c@{\hskip 3pt}c@{\hskip 3pt}c@{\hskip 2pt}}
a_1 & \ldots & a_n \\
b_1 & \ldots & b_n \\
\end{array}
\right]
$
and $k := 1$. \\
{\it Step 1}  & Put $y$ the maximal entry of $U$. \\
{\it Step 2a} & If $|y| = 0$ then put $i$ the minimal index of a row of $U$ containing $y$. \\
              & Put $T,x$ the output of $i \leftarrow T$ and delete $y$ from the end
                of row $i$ of $U$. \\
{\it Step 2b} & If $|y| = 1$ then put $j$ the minimal index of a column of $U$ containing $y$. \\
              & Put $T,x$ the output of $T \rightarrow j$ and delete $y$ from the end
                of column $j$ of $U$. \\
{\it Step 3}  & Put $a_k := x,b_k := y$ and $k := k + 1$. \\
{\it Step 4}  & If $k \leq n $ then go to Step 1 else output $S$. \\
\end{tabular}
\end{df}

\smallskip

The procedures $S \mapsto (T,U)$ and $(T,U) \mapsto S$ are both correct and
one is the inverse of the other. The proof of these facts can be viewed in section 4 of \cite{BonettiSenatoVenezia88} (more precisely, in Proposition 4.2 and 4.3). From this, Theorem \ref{correspondence} follows immediately.

The number of words in each plactic class can be computed in the
following way. Put $P = P_0 := \N$ and note that there is an
injective map $w \mapsto S$ from the set $\L$ into the set of
two-rowed arrays, where if $w:= x_1 \cdots x_n$ then:
\[
S :=
\left[
\begin{array}{@{\hskip 2pt}c@{\hskip 3pt}c@{\hskip 3pt}c@{\hskip 2pt}}
x_1 & \ldots & x_n \\
 1  & \ldots &  n  \\
\end{array}
\right]
\]
If $S \mapsto (T,U)$, from Definition \ref{df: SmapstoTU} it follows that
$T(w) = T$ since $P = P_0$. Moreover, if $\la\vdash n$ is the shape of $T$
then $U$ is a standard tableau of shape $\la$ whose entries are $1,2,\ldots,n$.
By Theorem \ref{correspondence} we get finally:

\begin{cor}
Let $L$ be a signed alphabet. Let $w\in\L$ and $\la$ be the shape
of the tableau $T := T(w)$. The number of words in the plactic class
$[w]$ is equal to the number of standard Young tableaux of shape $\la$.
\end{cor}

A natural question is to give the right  generalization of the so
called {\em symmetry Theorem} in the case of signed alphabets.
Consider the product alphabets $L \times P$ and $P \times L$.
Recall that they are ordered by the right lexicographic ordering
and $|(a,b)| := |a| + |b|$ for any pair $(a,b)$ belonging to one
of them. We can define an involution between the set of two-rowed
arrays on the alphabets $L,P$ and the corresponding set for $P,L$
simply as:
\[
S =
\left[
\begin{array}{@{\hskip 2pt}c@{\hskip 3pt}c@{\hskip 3pt}c@{\hskip 2pt}}
a_1 & \ldots & a_n \\
b_1 & \ldots & b_n \\
\end{array}
\right]
\mapsto
S' =
\left[
\begin{array}{@{\hskip 2pt}c@{\hskip 3pt}c@{\hskip 3pt}c@{\hskip 2pt}}
b'_1 & \ldots & b'_n \\
a'_1 & \ldots & a'_n \\
\end{array}
\right]
\]
where $\{a_1, \ldots, a_n\} = \{a'_1, \ldots, a'_n\}$ and
$\{b_1, \ldots, b_n\} = \{b'_1, \ldots, b'_n\}$.

\smallskip
\begin{df}\label{df: symmetry}
Let $L,P$ signed alphabets. We say that a two-rowed array $S$ {\em
has symmetry} if $S\mapsto (T,U)$ and $S'\mapsto (U,T)$.
\end{df}

\smallskip
The symmetry theorem \cite{Schutz63,Knuth70} states that in the non-signed case
($L = L_0,P = P_0$ or $L = L_1,P = P_1$) all two-rowed arrays have symmetry.
We generalize this result to a particular case in the following way:

\begin{prop}
\label{susy}
Assume that the alphabets satisfy the conditions $L_0 < L_1,P_0 < P_1$
or $L_1 < L_0,P_1 < P_0$ and let
$
S =
\left[
\begin{array}{@{\hskip 2pt}c@{\hskip 3pt}c@{\hskip 3pt}c@{\hskip 2pt}}
a_1 & \ldots & a_n \\
b_1 & \ldots & b_n \\
\end{array}
\right]
$
be a two-rowed array on $L,P$. If $|(a_i,b_i)| = 0$ for $i=1,2,\ldots,n$
then $S$ has symmetry.
\end{prop}

\smallskip
Before proving the theorem we need new notations. Assume $L_0 < L_1,
P_0 < P_1$ and let $S$ be a two-rowed array such that $|(a_i,b_i)| = 0$
that is $|a_i|$ = $|b_i|$ for all $i$. Since $b_1 \leq \ldots \leq b_n$,
note that $S$ {\em splits} as
$
S =
\left[
\begin{array}{@{\hskip 2pt}c@{\hskip 4pt}c@{\hskip 2pt}}
S_0 & S_1 \\
\end{array}
\right],
$
where $S_0$ is a two-rowed array on $L_0,P_0$ and $S_1$ on $L_1,P_1$.
Moreover, if $|a_i| = |b_i| = 0$ for $1 \leq i \leq k$ then clearly
$S_0$ has $k$ columns.

Suppose now that $S \mapsto (T,U)$ and let $\la \vdash n$ be the shape
of the Young tableaux $T,U$. Note that also $T$ {\em splits} into two tableaux
$T_0,T_1$ of content respectively $\{a_1,\ldots,a_k\},\{a_{k+1},\ldots,a_n\}$,
where $T_0$ has some shape $\mu \vdash k,\mu \subset \la$ and $T_1$ has the skew
shape $\la/\mu$. The same happens to the tableau $U$ and it is clear that
$S_0 \mapsto (T_0,U_0)$.

\begin{ex}
Let $L_0 = P_0 := \{1,2\},L_1 = P_1 := \{3,4,5,6\}$ and define:
\[
S :=
\left[
\begin{array}{cccccccc}
2 & 1 & 1 & 1 & 6 & 5 & 4 & 3  \\
1 & 2 & 2 & 2 & 3 & 4 & 5 & 6  \\
\end{array}
\right].
\]
Then, we have:
\[
T =
\begin{array}{cccc}
1 & 1 & 1 & 6 \\
2 & 4 & 5     \\
3             \\
\end{array},
\
U =
\begin{array}{cccc}
1 & 2 & 2 & 6 \\
2 & 4 & 5     \\
3         \\
\end{array}
\]
and hence:
\[
T_0 =
\begin{array}{ccc}
1 & 1 & 1 \\
2         \\
\end{array},
\
T_1 =
\begin{array}{cccc}
  &   &   & 6 \\
  & 4 & 5     \\
3             \\
\end{array},
\
U_0 =
\begin{array}{ccc}
1 & 2 & 2 \\
2      \\
\end{array},
\
U_1 =
\begin{array}{cccc}
  &   &   & 6 \\
  & 4 & 5     \\
3             \\
\end{array}.
\]
\end{ex}

Finally, for any partition $\la = (\la_1,\ldots,\la_r)$ we define:
\[
C_\la :=
\begin{array}{cccl}
1 & 2 & \ldots & \hskip 10pt \lambda_1 \\
1 & 2 & \ldots & \hskip 5pt  \lambda_2 \\
\vdots & \vdots                        \\
1 & 2 & \ldots & \lambda_r
\end{array}
\]
Clearly $C_\la$ is a Young tableau on the signed alphabet $L = L_1 :=
\{1,2,\ldots,\la_1\}$.

\medskip
\noindent{\it Proof of Proposition \ref{susy}.}
We argue for the case $L_0 < L_1,P_0 < P_1$ since the other case
can be recovered by conjugating the alphabets. Let $S \mapsto S'$,
where
$
S' =
\left[
\begin{array}{@{\hskip 2pt}c@{\hskip 3pt}c@{\hskip 3pt}c@{\hskip 2pt}}
b'_1 & \ldots & b'_n \\
a'_1 & \ldots & a'_n \\
\end{array}
\right],
$
and let $S \mapsto (T,U), S' \mapsto (T',U')$. Since $|a_i| = |b_i|$ and
therefore $|b'_i| = |a'_i|$ for any $i$, we have that the two-rowed arrays
$S,S'$ split respectively into $S_\alpha,S'_\alpha$ $(\alpha = 0,1)$.
Moreover, denote by $T_\alpha,U_\alpha,T'_\alpha,U'_\alpha$ the tableaux
obtained by splitting $T,U,T',U'$. Clearly $S_0$ maps to $S'_0$ under
the involution and we have that
$
S_0 \mapsto (T_0,U_0),
S'_0 \mapsto (T'_0,U'_0).
$
By the symmetry theorem it follows that $S_0$ has symmetry that is
$T'_0 = U_0,U'_0 = T_0$.

We claim that we have also $T'_1 = U_1,U'_1 = T_1$ and hence $S$ has symmetry.
Let $\lambda \vdash n$ and $\mu \subset \la$ be the shapes of the pairs
of tableaux respectively $T,U$ and $T_0,U_0$. Consider the tableau $C_\mu$
defined on an alphabet $\bar{L} := \bar{L}_1$ such that $\bar{L}_1 < L_1$ and
put $w := w(\tilde{C_\mu})$. If $w = \bar{a}_1 \cdots \bar{a}_k$ then we define
also
$
\bar{S}_1 :=
\left[
\begin{array}{@{\hskip 2pt}c@{\hskip 3pt}c@{\hskip 3pt}c@{\hskip 2pt}}
\bar{a}_1 & \ldots & \bar{a}_k \\
\bar{b}_1 & \ldots & \bar{b}_k \\
\end{array}
\right],
$
where $\bar{b}_1 < \ldots < \bar{b}_k$ are letters of an alphabet
$\bar{P} := \bar{P}_1$ such that $\bar{P}_1 < P_1$. Put
$
\bar{S} :=
\left[
\begin{array}{@{\hskip 2pt}c@{\hskip 4pt}c@{\hskip 2pt}}
\bar{S}_1 & S_1 \\
\end{array}
\right]
$
and let $\bar{S} \mapsto (\bar{T},\bar{U})$. The two-rowed array $\bar{S}_1$
clearly maps to a pair of tableaux of shape $\mu$ (the left one is $C_\mu$).
>From $\bar{L}_1 < L_1,\bar{P}_1 < P_1$ it follows that $\bar{T},\bar{U}$ have
shape $\la$ and are respectively the union of these tableaux with the tableaux
$T_1,U_1$ of skew shape $\la/\mu$. Since $\bar{S}_1,\bar{S}$ have symmetry
we conclude that the claim holds.
\qed

We remark that the condition in Proposition \ref{susy} is not
necessary. As an instance, whatever the signature is, the
two-rowed array
\[
\left[
\begin{array}{cccc}
3 & 4 & 1 & 2  \\
1 & 2 & 3 & 4
\end{array}
\right]
\]
has symmetry. To the best of our knowledge, proving a symmetry
Theorem taking into account a definition of symmetry more general
than Definition \ref{df: symmetry} and leading to the classical
statement as a particular case, is still an open problem.

\end{document}